\documentclass[12pt]{amsart}

\usepackage{amssymb}
\usepackage{mathrsfs}
\usepackage{xypic}
\usepackage{graphicx}
\usepackage{psfrag}
  
\theoremstyle{plain}
\newtheorem{theorem}{Theorem}[section]

\newtheorem{proposition}{Proposition}[section]
\newtheorem{corollary}{Corollary}[section]
\newtheorem*{Theorem}{Theorem}

\theoremstyle{definition}
\newtheorem{definition}{Definition}[section]

\theoremstyle{remark}

\newtheorem{example}{Example}[section]

\renewenvironment{proof}[1]{\vspace*{.1in}\noindent{\bf
Proof{#1}. \/}}{\qed\vspace{3ex}} 
		     
\newcommand{\Id}{\operatorname{Id}}

\newcommand{\Cay}{\operatorname{Cay}}
\newcommand{\Aut}{\operatorname{Aut}}

\newcommand{\bbR}{\mathbb R}

\newcommand{\cY}{\mathcal Y}
\newcommand{\cU}{\mathcal U}

\renewcommand{\>}{\rangle}

\catcode`\@=11
\def\@secnumfont{\bfseries}
\def\section{\@startsection{section}{1}%
  \z@{.7\linespacing\@plus\linespacing}{.5\linespacing}%
  {\normalfont\centering\bfseries}}
\def\subsection{\@startsection{subsection}{2}%
  \z@{.5\linespacing\@plus.7\linespacing}{-.5em}%
  {\normalfont\bfseries}}
\catcode`\@12

\title{On groups with Cayley graph isomorphic to a cube}
\author[C.~Hagemeyer]{Colin Hagemeyer}
\author[R.~Scott]{Richard Scott$^\dagger$}

\address{Department of Mathematics and Computer Science\\
Santa Clara University\\
Santa Clara, CA  95053}
\email{rscott@scu.edu}

\thanks{Both authors were supported by a Provost Office Grant from Santa Clara University}
\thanks{$\dagger$ Corresponding author}

\begin{document}

\begin{abstract}
We say that a group $G$ is a {\em cube group} if it is generated by a
set $S$ of involutions such that the corresponding Cayley graph $\Cay(G,S)$
is isomorphic to a cube.  Equivalently, $G$ is a cube group if it acts on
a cube such that the action is simply-transitive on the vertices and
the edge stabilizers are all nontrivial.  The action on the cube
extends to an orthogonal linear action, which we call the {\em
  geometric representation}.  We prove a combinatorial decomposition
for cube groups into products of $2$-element subgroup, and show that
the geometric representation is always reducible.
\end{abstract}

\maketitle

\section{Introduction}

Let $G$ be a group and let $S\subseteq G$ be a subset consisting of
involutions.  We say that the pair $(G,S)$ is a {\em cube group (of
rank $n$)} if the corresponding Cayley graph $\Cay(G,S)$ is
isomorphic to the $1$-skeleton of the $n$-cube.  Groups acting
on CAT($0$) cube complexes such that the action is simply-transitive
on vertices and has nontrivial edge stabilizers were studied in
\cite{Scott} as a natural generalization of right-angled Coxeter
groups.  Cube groups are precisely the finite groups in this class. 

The purpose of this note is to prove two theorems about cube groups.
The first is a product decomposition that implies the existence of a
type of ``boolean'' normal form. 

\begin{Theorem}
Let $(G,S)$ be a cube group of rank $n$.  For
each $s\in S$, let $\<s\>$ denote the subgroup generated by $s$.  Then
there exists an ordering $s_1,\ldots,s_n$ on the set $S$ such that  
\[G=\<s_1\>\<s_2\>\cdots\<s_n\>.\]
In particular, for any $g\in G$, there exist a unique choice of
$m_i\in\{0,1\}$, $i=1,\ldots,n$ such that  
\[g=s_1^{m_1}s_2^{m_2}\cdots s_n^{m_n}.\]
\end{Theorem}

The action of a cube group $G$ on $\Cay(G,S)$ extends canonically
to an isometric linear action on the full $n$-cube $[-1,1]^n$.  We
call this linear representation the {\em geometric representation} of
$G$.  The second theorem is a decomposition theorem for this
representation. 

\begin{Theorem}\label{Thm:reducible}
If $(G,S)$ is a cube group of rank $\geq 2$, then the geometric
representation is reducible.
\end{Theorem} 

Both of these theorems are consequences of the following general fact
about certain actions of $p$-groups, the proof of which is reminiscent
of one of the standard combinatorial proofs of Sylow's theorem. 

\begin{Theorem}
Let $X$ be a set with more than one element, and let $G$ be a
$p$-group acting on $X$.  If $G$ has a generating set such that every
element fixes some element in $X$, then $X$ has at least two orbits.   
\end{Theorem}


\section{Decorated graphs and  group presentations}
In this section, we described the relations among the generators of a
cube group in terms of a certain graph with involutions.  Let $S$ be a
finite set, and let $\Aut(S)$ denote its permutation group.  For each
$s\in S$, let $j_s\in\Aut(S)$ be an involution satisfying $j_s(s)=s$.
We can represent this data pictorially as follows.  Let $K_S$ denote
the complete graph with vertex set $S$. At each vertex $s$, draw an
arc between all pairs of edges $\{s,u\}$ and $\{s,v\}$ whenever
$j_s(u)=v$.  We shall refer to such a choice of involutions
$\Gamma=\{j_s\;|\; s\in S\}$ (and the resulting picture) as a  {\em
  decorated graph}. 

\begin{example}
For the decorated graph on $S=\{a,b,c,d,e\}$ shown in
Figure~\ref{fig:5-simplex}, the involutions $j_a$, $j_b$, $j_c$, and
$j_d$ are the transpositions $(bd)$, $(ac)$, $(bd)$, and $(ac)$,
respectively.  The involution $j_e$ is the product $(ac)(bd)$.
\begin{figure}[ht]
\begin{center}
\psfrag{a}{$a$}
\psfrag{b}{$b$}
\psfrag{c}{$c$}
\psfrag{d}{$d$}
\psfrag{e}{$e$}
\includegraphics[scale = .7]{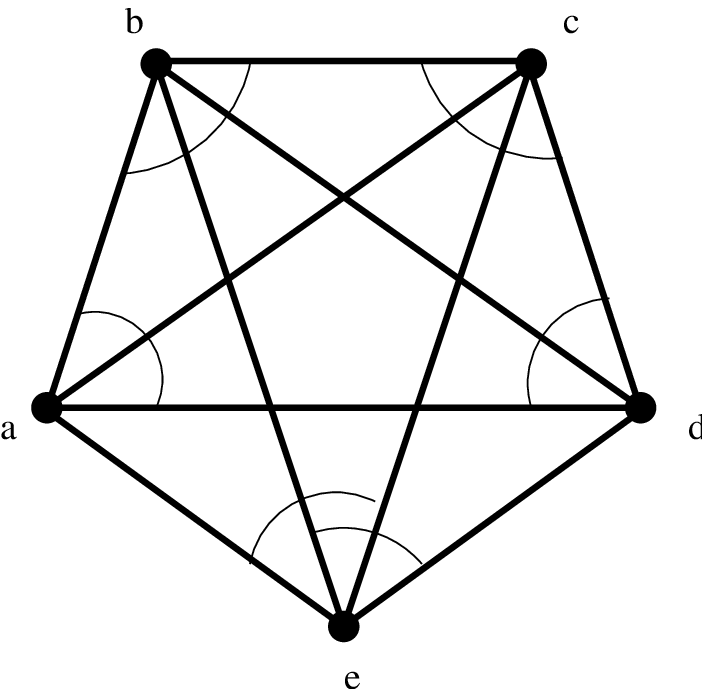}
\caption{\label{fig:5-simplex}}
\end{center}
\end{figure}
\end{example}

Given a decorated graph $\Gamma$ on $S$ and any two distinct elements
$s_1,s_2\in S$, one can define a sequence in $S$ inductively by
$s_{n+1}=j_{s_n}(s_{n-1})$.  Such a sequence will be called a {\em
    trajectory}.  A trajectory $s_1,s_2,\ldots$ is {\em $4$-periodic}
  if $s_n=s_{n+4}$ for all $n\geq 1$.  We shall say the decorated
  graph has {\em no holonomy} along a trajectory $s_1,s_2,\ldots$ if
  $j_{s_4}\circ j_{s_3}\circ j_{s_2}\circ j_{s_1}$ is the identity
  permutation in $\Aut(S)$.    

\begin{definition}
A decorated graph on $S$ is {\em admissible} if every trajectory is
$4$-periodic and has no holonomy along it.
\end{definition}

The $4$-periodicity condition simply means that the edges of a
decorated graph (with their connecting arcs) can be partitioned into
subsets of the form
\begin{itemize}
\item a $4$-cycle with arcs joining consecutive edges, 
\item an angle (two edges meeting at a vertex joined by an arc), or 
\item a single edge (with no connecting arcs touching it).  
\end{itemize}
For example, the decorated graph in Figure~\ref{fig:5-simplex}
satisfies $4$-periodicity since the edges can be partitioned as a
single $4$-cycle ($\Box abcd$), two angles ($\angle aec$, $\angle bed$), and
two single edges ($ac$ and $bd$).  The no-holonomy condition is more
difficult to verify but also holds for this decorated graph (for
example, along a trajectory $a,b,c,d,a,b,c,d,\ldots$ corresponding to
the $4$-cycle, we have $j_d\circ j_c \circ j_b \circ
j_a=(ac)(bd)(ac)(bd)=\Id$).

Given a decorated graph $\Gamma$ on $S$, we can define a group $W(\Gamma)$ 
by the presentation 
\[W(\Gamma)=\<s\in S\;|\;\mbox{$s^2=1$ for all $s\in S$ and $s_1s_2s_3s_4=1$
  for all trajectories $s_1,s_2,\ldots $}\>.\]
The following is a special case of the main result (Theorem~3.2) in
\cite{Scott}. 

\begin{theorem}\label{thm:SMT}
$(G,S)$ is a cube group if and only if there exists an admissible
  decorated graph $\Gamma$ on $S$ and an isomorphism
  $W(\Gamma)\rightarrow G$ that restricts to the identity on $S$.
\end{theorem} 

We refer the reader to \cite{Scott} for the full proof of this
theorem, but indicate here how one obtains the decorated graph
$\Gamma$ from the pair $(G,S)$.

Suppose $(G,S)$ is a cube group.  Recall that
the Cayley graph $\Cay(G,S)$ has vertex set $G$, and two vertices 
$g,g'$ are joined by a directed edge if $g'=gs$ for some $s\in S$.  Each
directed edge $(g,gs)$ is labeled by the element $s$, but since
$s^{-1}=s$, we can unambiguously label the undirected edge $\{g,gs\}$
by the element $s$ as well.  We shall therefore regard $\Cay(G,S)$ as
an undirected graph on $G$ with edges labeled by elements of $S$.
Recall that a presentation for $G$ can be obtained from the Cayley
graph by taking products of labels around all cycles in the graph.  

Since the Cayley graph for $G$ is isomorphic to a cube, cycles are
generated by cycles of length $4$, and any such $4$-cycle can be
translated by the (left) $G$-action (preserving labels) so that it
starts at the identity vertex $1$.  For a presentation, it therefore
suffices to consider the $4$-cycles touching the vertex $1$ (and the
trivial $2$-cycles obtained by repeating an edge).   At the vertex
$1$, there are precisely $n$ incident edges, labeled by the 
elements of $S$.  Any two elements $s_1,s_2\in S$, determine two such
edges and hence span a unique $4$-cycle ($2$-dimensional subface of
the cube).  Reading the  labels around this $4$-cycle, we obtain a
relation $s_1s_2s_3s_4=1$.   

To get the corresponding decorated graph $\Gamma$, we continue the
relation to obtain a $4$-periodic sequence
$s_1,s_2,s_3,s_4,s_1,\ldots$.  These are the trajectories of our
decorated graph $\Gamma$.  To obtain the involutions $j_s$, $s\in S$,
we simply consider all trajectories of the form $u,s,v,\ldots$ and
define $j_s(u)=v$.  One then needs to check that the resulting $j_s$'s
are well-defined involutions, and that there is no holonomy along
trajectories.  

\begin{example}\label{ex:d4}
Let $G$ be the dihedral group of order $8$, represented as the
permutation group $G=\{\Id, (13), (24), (13)(24), (12)(34), (14)(23),
(1234), (1432)\}$.  Letting $S=\{a,b,c\}$ where $a=(13)$,
$b=(12)(34)$, and $c=(14)(23)$, we obtain the $3$-cube as the Cayley
graph (Figure~\ref{fig:d4}, left-hand side), hence $(G,S)$ is a cube group of rank
$3$.  Looking at pairs of edges incident to the vertex $\Id$, we get the
trajectories $a,b,a,c,\ldots$, $b,a,c,a,\ldots$, $a,c,a,b,\ldots$,
$c,a,b,a,\ldots$, $b,c,b,c,\ldots$, and $c,b,c,b,\ldots$.  The
coresponding decorated graph on $S$ is therefore
$\Gamma=\{j_a,j_b,j_c\}$ (Figure~\ref{fig:d4}, right-hand side) where $j_a=(bc)$ and $j_b=j_c=\Id$, and we
have a presentation for $G$ of the form
\[W(\Gamma)=\<a,b,c\;|\; a^2=b^2=c^2=1, abac=1, bcbc=1\>.\]
\begin{figure}[ht]
\begin{center}
\psfrag{a}{$a$}
\psfrag{b}{$b$}
\psfrag{c}{$c$}
\psfrag{Id}{$\Id$}
\psfrag{(12)(34)}{$(12)(34)$}
\psfrag{(13)}{$(13)$}
\psfrag{(24)}{$(24)$}
\psfrag{(14)(23)}{$(14)(23)$}
\psfrag{(1234)}{$(1234)$}
\psfrag{(1432)}{$(1432)$}
\psfrag{(13)(24)}{$(13)(24)$}
\includegraphics[scale = .6]{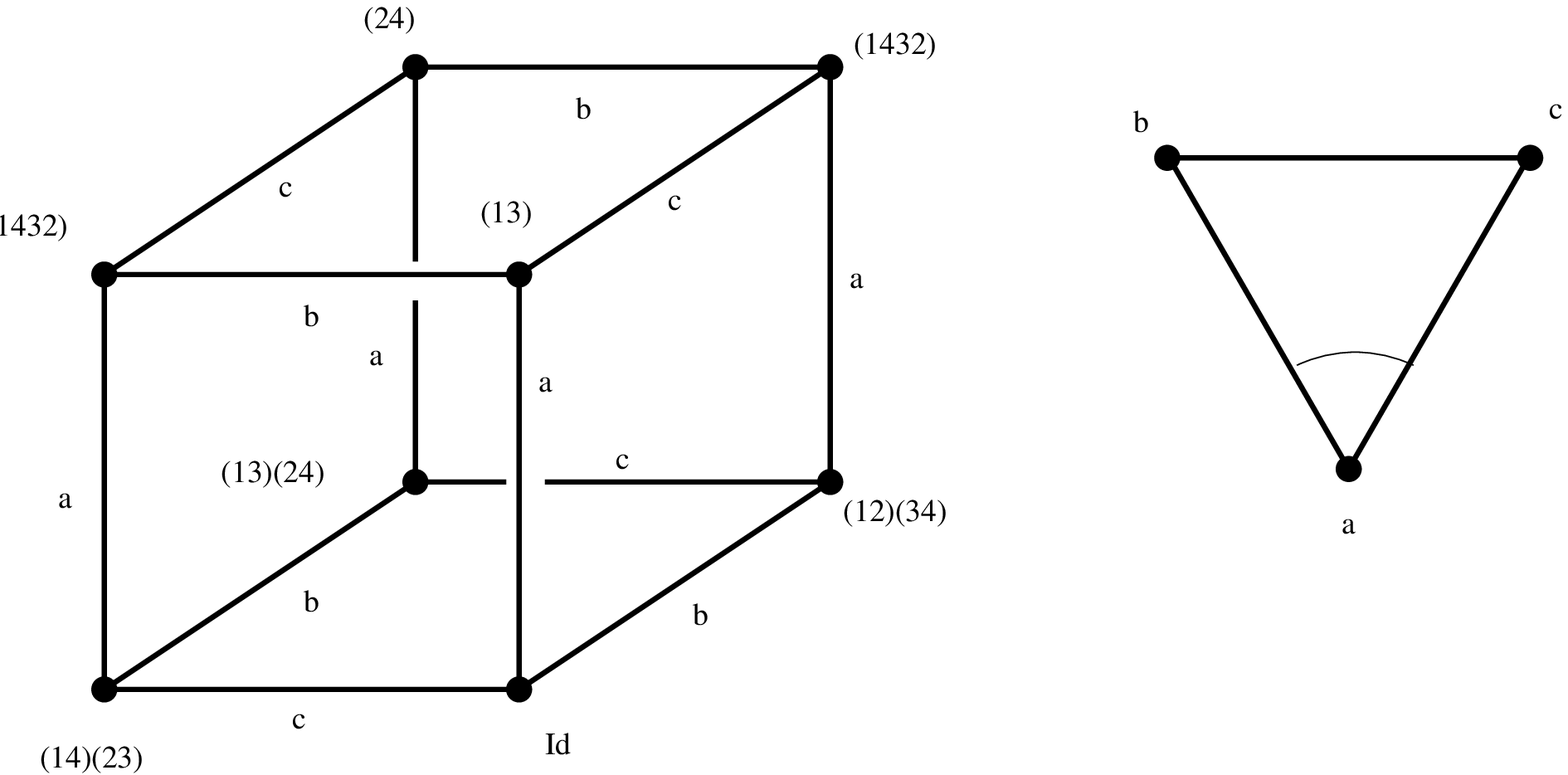}
\caption{\label{fig:d4}}
\end{center}
\end{figure}
\end{example}

\section{Product decompositions of cube groups}

Given a cube group $(G,S)$ any subset $T\subseteq S$ generates a
subgroup of $G$ which we denote by $G_T$.  We will call this subgroup
a {\em standard subgroup} if $(G_T,T)$ is also a cube group.  

\begin{example}
Let $(G,S)$ be the dihedral group in Example~\ref{ex:d4}.  The
subgroup generated by $T=\{b,c\}$ is a standard subgroup since
$(G_T,T)$ has Cayley graph isomorphic to a square.  On the other hand,
if $T=\{a,b\}$, then $(G_T,T)$ is not a standard subgroup.  In this
case $G_T$ is the entire group ($c=aba$) and $T$ has only two
elements, hence the Cayley graph $\Cay(G_T,T)$ will be an $8$-cycle,
which is not isomorphic to a cube. 
\end{example}

In this section we describe how to decompose a cube group into
products of standard subgroups.  We use the fact that there is a
natural action of the group $G$ on its generating set $S$.

\begin{proposition}\label{prop:permrep}
Let $(G,S)$ be a cube group with corresponding decorated graph
$\Gamma=\{j_s\;|\;s\in S\}$.  Then the map $S\rightarrow\Aut(S)$ given
by $s\mapsto j_s$ extends uniquely to a homomorphism
$j:G\rightarrow\Aut(S)$.
\end{proposition}

\begin{proof}{}
One needs only check that the $j_s$'s satisfy equations corresponding
to the relations in the presentation $W(\Gamma)$.  This is precisely
the requirement that the $j_s$'s be involutions and that $\Gamma$ have
no holonomy along trajectories.
\end{proof} 

We shall call $j:S\rightarrow\Aut(S)$ the {\em permutation
representation} for $(G,S)$.  It defines an action of $G$ on the set
$S$ by $g\cdot s=j_g(s)$ for all $g\in G$ and $s\in S$.  Invariant
subsets (i.e., unions of orbits) of this action give rise to product
decompositions of $(G,S)$ with respect to standard subgroups. 

\begin{proposition}\label{prop:inv-subset}
Let $(G,S)$ be a cube group and let $T\subseteq S$ be a $G$-invariant
subset of $S$.  Then $(G_T,T)$ is a standard subgroup and, moreover,
we have a product decomposition
\[G=G_TG_{S-T}\]
(meaning any element $g\in G$ can be written uniquely in the
form $g=g_1g_2$ where $g_1\in G_T$ and $g_2\in G_{S-T}$).
\end{proposition}

\begin{proof}{}
For any subset $T\subseteq S$, let $X_T$ denote the $|T|$-dimensional
subcube of $\Cay(G,S)$ containing the edges incident to $1$ that are
labeled by $t\in T$.  To show that $G_T$ is a standard subgroup, it
suffices to show that all of the labels on the edges of $X_T$ are in
the subset $T$, since then $X_T$ will coincide with $\Cay(G_T,T)$.
For this, it is enough to show that the labels around any $4$-cycle in
$X_T$ are always in $T$.   Since $T$ is $G$-invariant, any trajectory
starting $t_1,t_2,j_{t_2}(t_1),\ldots$ with $t_1,t_2\in T$ will have
all terms in $T$, hence if a $4$-cycle in $\Cay(G_T,T)$ has $2$
consecutive edges in $T$, it will have all edges in $T$.  It follows
that all $4$-cycles in $X_T$ incident to the vertex $1$ have edges
labeled by elements of $T$, and then by induction on the distance to
$1$ that all $4$-cycles in $X_T$ are labeled by elements of $T$.

Since $S-T$ will also be $G$-invariant, we have $2$ standard subgroups
$G_T$ and $G_{S-T}$.  The subcubes $X_T$ and $X_{S-T}$ will intersect
only in the vertex $1$, hence the two subgroups $G_T$ and $G_{S-T}$
have trivial intersection.  To prove the product
decomposition, it is enough to show that any product $st$ with
$s\in S-T$ and $t\in T$ can be rewritten at $st=t's'$ with $t'\in T$
and $s'\in S-T$.   Consider the trajectory $s,t,s',t',\ldots$ starting
with $s$ and $t$.  Then $s'=j_t(s)$ and
$t'=t_{s'}(t)$.  Since $S-T$ is $G$-invariant and $s\in
S-T$, we know $s'\in S-T$, and since $T$ is $G$-invariant and $t\in
T$, we know $t'\in T$.  On the other hand, the relation coming from
this trajectory is $sts't'=1$, or (since generators are all
involutions) $st=t's'$.    
\end{proof}

It turns out that for any cube group of rank at least $2$, there is
always a proper nontrivial invariant subset.  This is the key
technical result of this paper and uses the same counting argument
found in one of the standard proofs of the first Sylow theorem (see
\cite{Wielandt}).

\begin{theorem}\label{thm:2-orbits}
Let $G$ be a $p$-group acting on a set $X$ with $|X|\geq 2$.  Assume
further that there is a generating subset $S\subseteq G$ such that for
all $s\in S$ there exists an $x\in X$ such that $s\cdot x=x$.  Then
$X$ has at least two $G$-orbits.
\end{theorem}

\begin{proof}{}
Assume, on the contrary that the action is transitive on $X$.  Since
any subgroup of $G$ is also a $p$-group, the order of any orbit must
also be a power of $p$, hence $|X|=p^n$ for some $n$.  Let $\cU$
denote the set of all subsets of $X$ of size $p^{n-1}$.  Then 
\[|\cU| = \binom{p^n}{p^{n-1}} = p\prod_{j=1}^{p^{n-1}-1}\frac{p^n-j}{p^{n-1}-j}=
p\prod_{j=1}^{p^{n-1}-1}\frac{p^{n-\nu_p(j)}-j/p^{\nu_p(j)}}{p^{n-1-\nu_p(j)}-j/p^{\nu_p(j)}}\]
where $\nu_p(j)$ denotes the $p$-adic valuation of $j$.  Since none of
the terms in the product on the right have any factors of $p$, we can
write $|\cU|=pm$ for some integer $m$ relatively prime to $p$.  

Now consider the induced action of $G$ on $\cU$.  Since all orbits of this
action must have order a power of $p$, and since the sum of these
orders must be $pm$, there must be at least one orbit of size exactly
$p$.  Let $\cY$ be such an orbit.  Then $\cY$ consists of $p$ subsets
of $X$ of size exactly $p^{n-1}$, and since the action of $G$ on $X$
is transitive, these subsets form a partition of $X$.  Given $x\in X$,
we let $[x]\in\cY$ denote its equivalence class.  Now, given $s\in S$,
let $\<s\>$ denote the subgroup it generates, and let $\cY^{\<s\>}$
denote the elements of $\cY$ fixed by $\<s\>$.  Since there exists an
$x\in X$ such that $s\cdot x=x$, we know that $s\cdot[x]=[x]$, hence
$\cY^{\<s\>}$ is not empty.  On the other hand, since $\<s\>$ is a
$p$-group, we have 
\[|\cY^{\<s\>}|\equiv|\cY|\bmod p\equiv p \bmod p.\]
It follows that $\cY^{\<s\>}=\cY$, so $s$ fixes $\cY$.  Since this
holds for all $s$ in the generating set $S$, the action of $G$ on
$\cY$ must be trivial.  But this contradicts transitivity of the
action of $G$ on $X$
\end{proof}

To apply this theorem (in the case $p=2$) to our setting, we suppose
$(G,S)$ is a cube group with $|S|\geq 2$ and take $X=S$.  The action
of $G$ on $S$ has the property that $s\cdot s=j_s(s)=s$, hence every
generator fixes some element.  By Theorem~\ref{thm:2-orbits} it
follwos that the action has at least two orbits.  Combining this with
the previous proposition, we obtain our first theorem from the
introduction.  

\begin{corollary}
Let $(G,S)$ be a cube group of rank $n$.  For each $s\in S$, let
$\<s\>$ denote the subgroup generated by $s$.  Then there exists an
ordering $s_1,\ldots,s_n$ on the set $S$ such that   
\[G=\<s_1\>\<s_2\>\cdots\<s_n\>.\]
\end{corollary}

\begin{proof}{}
The proof is by induction on the rank of $S$.  If $S=\{s\}$, then
$G=\<s\>$.  In general, assume the decomposition holds for all cube
groups of rank $<n$, and let $(G,S)$ be a cube group of rank $n$.
By Theorem~\ref{thm:2-orbits}, there exists a proper, nontrivial
$G$-invariant subset $T\subseteq S$.  The groups $(G_T,T)$ and
$(G_{S-T},S-T)$ are cube groups of rank $<n$, say $k$ and $n-k$,
respectively.  Hence by induction, there exists an ordering
$s_1,\ldots,s_k$ of the elements in $T$, and an ordering
$s_{k+1},\ldots,s_n$ of the elements in $S-T$, such that 
\[G_T=\<s_1\>\<s_2\>\cdots\<s_k\>\;\;\mbox{and}\;\;
G_{S-T}=\<s_{k+1}\>\<s_{k+2}\>\cdots\<s_n\>.\]
The result then follows from Proposition~\ref{prop:inv-subset}.
\end{proof}

\begin{example}
If $(G,S)$ is the dihedral group of Example~\ref{ex:d4}, then $S$
breaks up into $G$-orbits $\{a\}$ and $T=\{b,c\}$, and $T$
breaks up into $G_T$-orbits $\{b\}$, and $\{c\}$.  Hence,
\[G=\<a\>\<b\>\<c\>\]
and we can take $a,b,c$ as our ordering on $S$.  In other words, 
the elements of $G$ can be listed as $G=\{1,a,b,c,ab,ac,bc,abc\}$.  On the
other hand, the ordering $b,a,c$ does not respect the orbit structure and 
$\{1,b,a,c,ba,bc,ac,bac\}$ will be a proper subset of $G$ (since
$ba=ac$ and $bac=a$).
\end{example}

\begin{example}
For a more complex example, consider the decorated graph in
Figure~\ref{fig:8-simplex} (the missing edges correspond to
$2$-periodic trajectories, hence commuting generators).  
\begin{figure}[ht]
\begin{center}
\psfrag{a}{$a$}
\psfrag{b}{$b$}
\psfrag{c}{$c$}
\psfrag{d}{$d$}
\psfrag{e}{$e$}
\psfrag{f}{$f$}
\psfrag{g}{$g$}
\psfrag{h}{$h$}
\psfrag{S}{$S$}
\psfrag{T1}{$T_1$}
\psfrag{T2}{$T_2$}
\psfrag{T11}{$T_{1,1}$}
\psfrag{T12}{$T_{1,2}$}
\psfrag{T21}{$T_{2,1}$}
\psfrag{T22}{$T_{2,2}$}
\includegraphics[scale = .55]{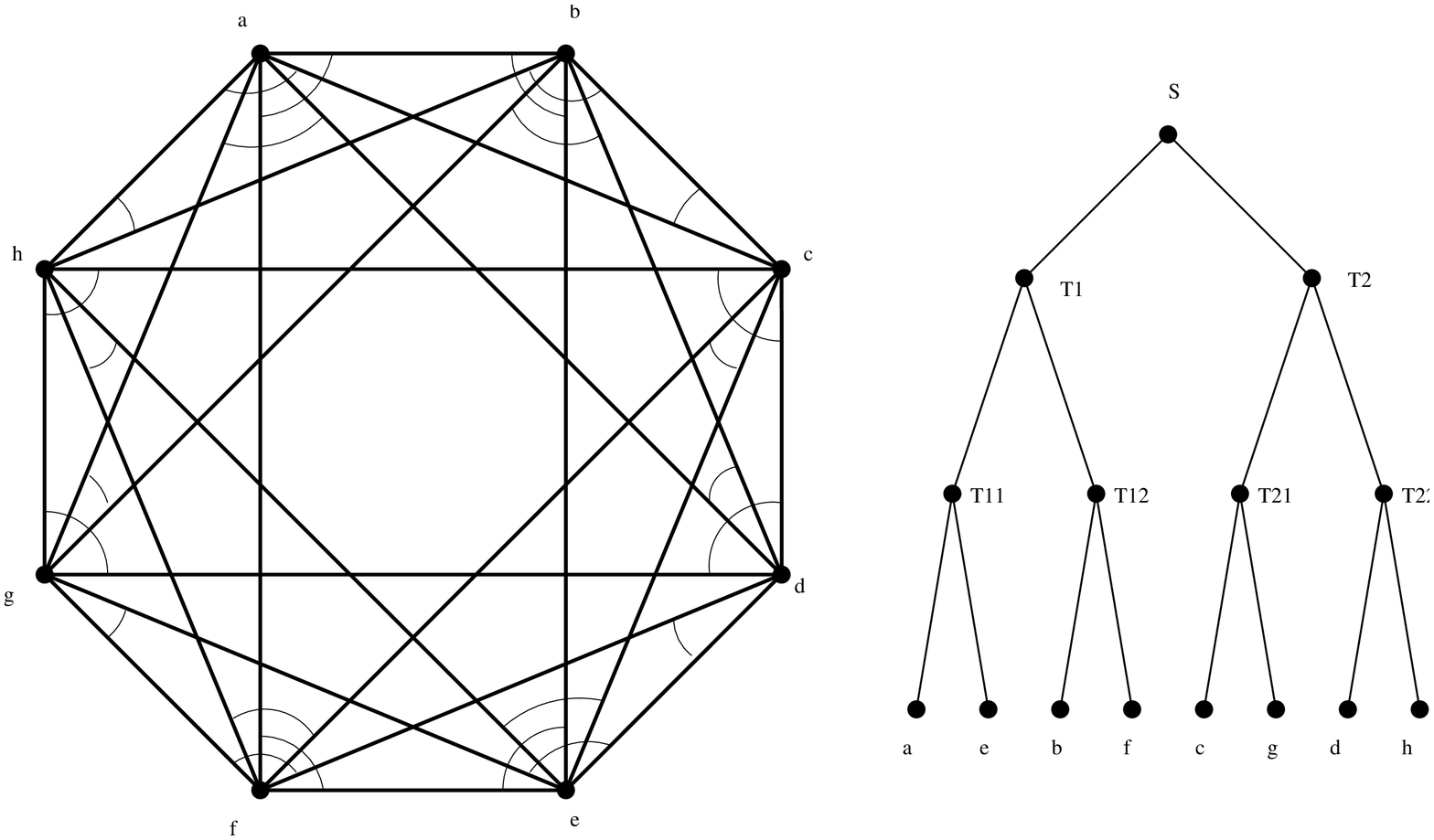}
\caption{\label{fig:8-simplex}}
\end{center}
\end{figure}
One can show that this graph is admissible, hence defines a cube group
$(G,S)$ of order $2^8=256$.  The $G$-orbits are $T_1=\{a,b,e,f\}$ and 
$T_2=\{c,d,g,h\}$.  The $G_{T_1}$-orbits are $T_{1,1}=\{a,e\}$,
$T_{1,2}=\{b,f\}$, and the $G_{T_2}$-orbits are $T_{2,1}=\{c,g\}$,
$T_{2,2}=\{d,h\}$.  Finally, all of these orbits break up into
singleton orbits, so we can take as our decomposition
\[G = G_{T_1}G_{T_2} = G_{T_{1,1}}G_{T_{1,2}}G_{T_{2,1}}G_{T_{2,2}} =
\<a\>\<e\>\<b\>\<f\>\<c\>\<g\>\<d\>\<h\>.\]
In fact any planar representation of the orbit tree shown in
Figure~\ref{fig:8-simplex} will give rise to a different product
decomposition by reading off the final nodes from left to right. 
\end{example}
 
\section{The geometric representation}
In this section we describe the geometric representation of a cube
group $(G,S)$ arising from the left action of $G$ on $\Cay(G,S)$.  We
use the fact that the $n$-cube is rigid in the sense that
any (graph) automorphism of its $1$-skeleton is the restriction of an
isometry of the entire cube.  

Let $(G,S)$ be a cube group, and let $\bbR^S$ denote the finite
dimensional real Euclidean space with standard basis $\{e_s\;|\;s\in
S\}$.  There is a natural identification of $\Cay(G,S)$ with
the standard cube $[-1,1]^S\subseteq\bbR^S$ given as follows.  Since
$\Cay(G,S)$ is isomorphic to the $n$-cube with $n=|S|$, we can index
the vertices using subsets of $S$.  More precisely, we let
$g_{\emptyset}$ be the identity vertex.  For any subset $T\subseteq S$,
there is a unique minimal subcube containing the vertex
$g_{\emptyset}$ and the vertices $s$ for all $s\in T$.  We let $g_T$ denote
the vertex opposite $g_{\emptyset}$ in this subcube.  The embedding of
$\Cay(G,S)$ into $\bbR^S$ is then given by the mapping 
\[g_T\mapsto v_T:=\sum_{s\not\in T}e_s-\sum_{s\in T}e_s\]
for all $T\subseteq S$.  It follows from the way we indexed the
elements of $G$ that the edges in the Cayley graph will map precisely
to the edges in the $1$-skeleton of the cube $[-1,1]^S$.  
For the remainder of the paper, we shall identify
$\Cay(G,S)$ with the $1$-skeleton of the cube $[-1,1]^S$ via this
embedding.   

Let $(G,S)$ be a cube group.  Any element $g\in G$ determines an
automorphism of the Cayley graph $\Cay(G,S)$, hence an isometry of the
cube $[-1,1]^S$.  Any such isometry is the restriction of a unique
orthogonal linear transformation $\rho_g:\bbR^S\rightarrow\bbR^S$.
The resulting homomorphism $\rho:G\rightarrow GL(\bbR^S)$ taking $g$
to $\rho_g$ will be called the {\em geometric representation} of the
pair $(G,S)$.  Since $G$ acts simply-transitively on itself and $G$ is
identified with the vertices of the cube $[-1,1]^S$, the geometric 
representation is obviously faithful.

The geometric representation can be described explicitly in terms of
the permutation representation $j:G\rightarrow\Aut(S)$.  Given a
product representation $g=s_ks_{k-1}\cdots s_1$ for $g$ in terms of
generators $s_1,s_2,\ldots 
s_k\in S$ and given any element $t\in S$, we let $n(g,t)$ denote the
cardinality of the set 
\[\{i\in[1,k-1]\;|\; s_{i+1}=j_{s_{i}}j_{s_{i-1}}\cdots j_{s_1}(t)\}.\]
We then obtain the following formula for $\rho$. 

\begin{proposition}
Let $(G,S)$ be a cube group and let $\Gamma=\{j_s\;|\;s\in S\}$ be the
corresponding decorated graph.  For any $g\in G$, we let
$j_g\in\Aut(S)$ denote the image of $g$ under the homomorphism $j$.
The linear transformation $\rho_g:\bbR^S\rightarrow\bbR^S$ is then
given by 
\[\rho_g(e_t)=(-1)^{n(g,t)}e_{j_g(t)}\;\;\mbox{for all $t\in S$}.\]
(In particular, $n(g,s)$ is independent of the product representation
of $g$ in terms of generators.)
\end{proposition} 

\begin{proof}{}
First we prove that for any $s\in S$, we have 
\[\rho_s(e_t)=\left\{\begin{array}{ll}
-e_t & \mbox{if $t=s$}\\
e_{j_s(t)} & \mbox{if $t\neq s$.}\end{array}\right.\]
For each $t\in S$, let $F_t^{+}$ (respectively $F_t^-$) denote the
facet (codimension-$1$ face) of the cube $[-1,1]^S$ with barycenter
$e_t$ (resp., $-e_t$).  The map $\rho_s$ must take the $n$ facets
incident to the vertex $1$ bijectively to the $n$ facets incident to the
vertex $s$, permuting those facets incident to {\em both} $1$ and $s$.
The facets incident to $1$ are precisely the ``positive''
ones $\{F_t^+\;|\; t\in S\}$, and the facets incident to $s$ are
$\{F_t^+\;|\; t\neq s\}\cup\{F_s^-\}$.  So $\rho_s$ must map $F_s^+$
to $F_s^-$, and hence for $t=s$ we have $\rho_s(e_t)=-e_t=-e_{j_s(t)}$.  

Now suppose $t\neq s$, then the facet $F_t^+$ is the unique facet
containing the $(n-1)$ edges incident to the vertex $1$ that are labeled by
the set $\{u\;|\; u\neq t\}$.  Equivalently, $F_t^+$ is the unique facet 
containing the $(n-1)$ edges incident to the vertex $s$ that are labeled by
the set $\{j_s(u)\;|\; u\neq t\}$ (see Figure~\ref{fig:facet}).  
\begin{figure}[ht]
\begin{center}
\psfrag{s}{$s$}
\psfrag{ss}{$\mathbf s$}
\psfrag{1}{$\mathbf 1$}
\psfrag{t}{$t$}
\psfrag{u}{$u$}
\psfrag{u'}{$u'$}
\psfrag{j_s(t)}{$j_s(t)$}
\psfrag{j_s(u)}{$j_s(u)$}
\psfrag{j_s(u')}{$j_s(u')$}
\psfrag{F}{$\mathbf F_t^+$}
\includegraphics[scale = .6]{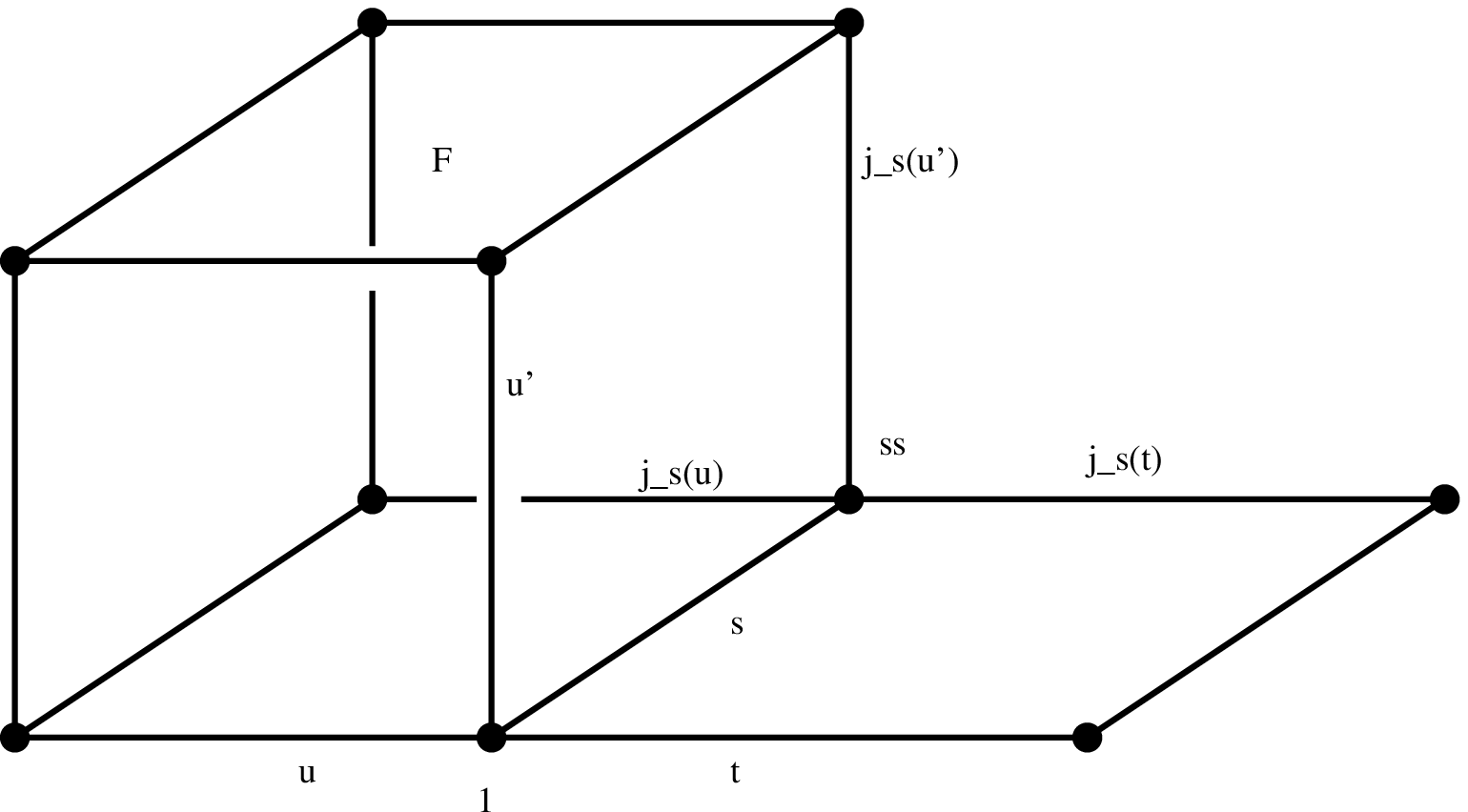}
\caption{\label{fig:facet}}
\end{center}
\end{figure}
Since
$\rho_s$ restricts to the left action of $s$ on the Cayley graph and
this action preserves edge labels, it follows that $\rho_s(F_s^+)$ is
the unique facet spanned by the edges incident to the vertex $s$ that
are labeled by the set $\{u\;|\; u\neq t\}$.  This is the same
as the facet spanned by the edges incident to the vertex $1$ that
are labeled by the set $\{j_s(u)\;|\; u\neq t\}$, or equivalently, the
set $\{u\;|\; u\neq j_s(t)\}$.  But this is precisely the facet
$F_{j_s(t)}^+$.  Restricting $\rho_s$ to the barycenters of these
facets, we then have $\rho_s(e_t)=e_{j_s(t)}$.

Finally, suppose $g=s_k\cdots s_2 s_1$.  Then 
\begin{align*}
\rho_g(e_t) & = \rho_{s_k}\cdots\rho_{s_2}\rho_{s_1}(e_t)\\
            & = \rho_{s_k}\cdots\rho_{s_2}((-1)^{n(s_1,t)}e_{j_{s_1}(t)})\\
            & =\rho_{s_k}\cdots\rho_{s_3}((-1)^{n(s_2s_1,t)} e_{j_{s_2}j_{s_1}(t)})\\
& \hspace{.5in}\vdots \\
            & =(-1)^{n(s_k\cdots s_1,t)}e_{j_{s_k}\cdots j_{s_2}j_{s_1}(t)}\\
            & =(-1)^{n(g,t)}e_{j_g(t)}
\end{align*}
as desired. 
\end{proof}
   
It follows immediately from this formula for the geometric
representation $\rho$ that if $T\subseteq S$ is a $G$-invariant subset of
$S$, then $\bbR^T\subseteq \bbR^S$ is an invariant subspace for
$\rho$.  By Theorem~\ref{thm:2-orbits}, $S$ always admits a proper
nontrivial $G$-invariant subset, hence we have our second theorem from
the introduction.

\begin{corollary}
If $(G,S)$ is a cube group, then the geometric representation $\rho$ 
is reducible. 
\end{corollary}

\end{document}